\documentclass{amsart}

\swapnumbers \theoremstyle{plain}
\newtheorem{thm}{Theorem}[section]
\newtheorem{lem}[thm]{Lemma}

\newtheorem{lem-defn}[thm]{Lemma and Definition}

\theoremstyle{definition}\newtheorem{rem}[thm]{Remark}
\newtheorem{defn}[thm]{Definition}

\theoremstyle{definition}

\newcommand{\Cal}{\mathcal}

\newcommand{\R}{\mathbb{R}}

\newcommand{\C}{\mathbb{C}}
\newcommand{\D}{\mathbb{D}}
\newcommand{\N}{\mathbb{N}}

 \DeclareMathOperator{\im}{Im}
\DeclareMathOperator{\ke}{Ker} 
 \DeclareMathOperator{\rank}{rank}
 
\DeclareMathOperator{\End}{End}

 \numberwithin{equation}{section}

\usepackage{soul,cancel,url}
\usepackage{amsfonts}
\usepackage{hyperref}
\usepackage{color}
\usepackage{amsmath}
    \usepackage{amsxtra}
    \usepackage{amstext}
    \usepackage{amssymb}
    \usepackage{latexsym}

\title [On the Jordan structure of holomorphic matrices]{On the Jordan structure of holomorphic matrices}
\author{J\"urgen Leiterer}
\address{Institut f\"ur Mathematik \\
Humboldt-Universit\"at zu Berlin \\Rudower Chaussee 25\\D-12489 Berlin , Germany}
\email{leiterer@mathematik.hu-berlin.de}

\date{}

\begin{document}

\begin{abstract} Let $X$ be an open subset of $\C^N$, and let $A$ be an $n\times n$ matrix of holomorphic functions on $X$. We call a point $\xi\in X$ {\bf Jordan stable} for $A$ if $\xi$ is not a splitting point of the eigenvalues of $A$ and, moreover, there is a neighborhood $U$ of $\xi$ such that, for each $1\le k\le n$, the number of Jordan blocks of size $k$ in the Jordan normal forms of $A(\zeta)$ is the same for all  $\zeta\in U$. H. Baumg\"artel \cite[S 3.4]{B4} proved that there is a nowhere dense closed analytic subset of $X$, which contains all points of $X$ which are not Jordan stable for $A$.
We give a new proof of this result. This proof has the advantage that the result can be obtained in a more precise form, and with some estimates. Also, this proof applies to arbitrary, possibly non-smooth, complex spaces $X$.
\end{abstract}

\maketitle
\section{Introduction}

Let $X$ be an open subset of $\C^N$, and let $A$ be an $n\times n$ matrix of holomorphic functions on $X$.

A  point $\xi\in X$ is called a {\bf splitting point of the eigenvalues} of $A$ if, for each neighborhood $U\subseteq X$ of $\xi$, there is a point $\zeta\in U$ such that $A(\zeta)$ has more eigenvalues than $A(\xi)$ (see Lemma \ref{15.1.17'} for an equivalent definition). It
 is well-known (cp. Remark \ref{24.11.16-}) that the set of all splitting points of the eigenvalues of $A$ is a nowhere dense closed analytic subset of $X$.\footnote{$Y\subseteq X$ is called a closed analytic subset of $X$ if, for each point $\xi\in X$, there exist a neighborhood $U\subseteq X$ of $\xi$ and holomorphic functions $f_1,\ldots,f_\ell$ on $U$ such that $Y\cap U=\{f_1=\ldots=f_\ell =0\}$. For $N=1$ this means that $Y$ is a closed discrete subset of $X$.}

We call a point $\xi\in X$ {\bf Jordan stable} for $A$ if $\xi$ is not a splitting point of the eigenvalues of $A$ and, moreover, there is a neighborhood $U$ of $\xi$ such that, for each $1\le k\le n$, the number of Jordan blocks of size $k$ in the Jordan normal forms of $A(\zeta)$ is the same for all  $\zeta\in U$ (see Definition \ref{6.2.17} for equivalent conditions). Denote by $\mathrm{Jst\,}A$ the set of all Jordan stable points of $A$.

H. Baumg\"artel  proved that $X\setminus\mathrm{Jst\,}A$ is contained in some  nowhere dense closed analytic subset of $X$, see \cite{B1}, \cite[Kap. V,\S 7]{B2}, \cite[5.7]{B4} for $N=1$, and \cite{B3}, \cite[S 3.4]{B4} for arbitrary $N$.

In the present paper, we give a new proof for Baumg\"artel's theorem, and also for the analyticity of the set of splitting points of the eigenvalues. These proofs have the advantage that the results can be obtained in a more precise form, and with some estimates. For example (Theorem \ref{1.12.16'}):

The set $X\setminus\mathrm{Jst\,}A$ is not only {\em contained} in a nowhere dense closed analytic subset of $X$, but it is itself such a set.
Moreover, there exist finitely many holomorphic functions $f_1,\ldots,f_n:X\to \C$ such that
\[X\setminus\mathrm{Jst\,}A=\big\{f_1=\ldots=f_\ell=0\big\}\] and, for some constants $K,k\in \N^*$  depending only on $n$ (and not on $X$ and $A$),
\[
\big\vert f(\zeta)\big\vert\le K\big(1+\Vert A(\zeta)\Vert\big)^k\quad\text{for all}\quad\zeta\in X.
\]This implies:

 \vspace{2mm}-- If $X$ is the open unit disk in $\C$ and $A$ is bounded, then  $X\setminus\mathrm{Jst\,}A$ satisfies the Blaschke condition.

 -- If $X=\C^N$ and the elements of $A$ are holomorphic polynomials, then $X\setminus\mathrm{Jst\,}A$ is affine algebraic. For $N=1$ this means that $X\setminus\mathrm{Jst\,}A$ is finite.

\vspace{2mm}

If $X$ has a $\Cal C^0$ boundary,   corresponding results are obtained for functions which admit a continuous extension to the boundary of $X$ (Section \ref{4.12.16'}).

Also, our proof applies to arbitrary, possibly non-smooth, complex spaces $X$.

\section{Notation}\label{22.2.15}

$\N$ denotes the set of natural numbers including $0$. $\N^*=\N\setminus\{0\}$.

If $n,m\in \N^*$, then  $\mathrm{Mat}(n\times m,\C)$  denotes the space of complex $n\times m$ matrices ($n$ rows and $m$ columns), and
$\mathrm{GL}(n,\C)$  denotes the  group of invertible elements of $\mathrm{Mat}(n\times n,\C)$.

For $\Phi\in \mathrm{Mat}(n\times m,\C)$, we denote by $\ke \Phi$,  $\im \Phi$, $\rank \Phi$ and $\Vert\Phi\Vert$  the kernel, the image, the rank and the operator norm (as a linear map between the Euclidean spaces $\C^m$ and $\C^n$) of $\Phi$, respectively.

The unit matrix in $\mathrm{Mat}(n\times n,\C)$ will be denoted by $I_n$ or simply by $I$. If $\Phi\in\mathrm{Mat}(n\times n,\C)$ and $\lambda\in\C$, then, instead of  $\lambda I_n-\Phi$ we write also $\lambda -\Phi$.

By a complex space we always mean a {\em reduced} complex space in the sense of, e.g.,  \cite{GR}, which is the same as an {\em analytic} space in the sense of,  e.g.,  \cite{L}.
 For example, each complex manifold and each analytic subset of a complex manifold is a complex space.

By an {\em irreducible} complex space we  mean a {\em globally} irreducible complex space, i.e., a complex space, for which the manifold of smooth points is connected, see, e.g., \cite[Ch. V.4.5]{L} or \cite[Ch. 9, \S 1]{GR}. For example, each connected complex manifold is an irreducible complex space.

If we say ``$\lambda_1,\ldots,\lambda_m$ are the eigenvalues of a matrix'' (or the zeros of a polynomial), then we mean this always  \underline{not counting multiplicities} (hence, then  $\lambda_i\not=\lambda_j$ if $i\not=j$).

\section{Splitting points of the zeros of monic polynomials}\label{7.1.17}

\begin{defn}\label{10.11.16'}
 By a complex polynomial  we  mean  a function $p:\C\to \C$ of the form $p(\lambda)=p_0+p_1\lambda+\ldots+p_n\lambda^n$,
where $n\in\N$ and $p_0,\ldots, p_n\in \C$. If $p_n=1$, $p$ is called {\bf monic}. The map from $\C$ to $\C$ which is identically zero will be called the {\bf zero polynomial}.

 If  $n\in  \N$, then we  denote by $\Cal P_n$ the complex vector space, which consists  of all complex polynomials of degree $\le n$ and the zero polynomial.

Now let $X$ be a topological space, $n\in\N^*$, and $P:X\to \Cal P_n$  a continuous map, all values of which are monic and of degree $n$.
Then
  $\xi\in X$ is called a  {\bf splitting point of the zeros of $P$} if, for each neighborhood $U$ of $\xi$, there exists $\zeta\in U$ such that  $P(\zeta)$ has more zeros than $P(\xi)$ (not counting multiplicities).
\end{defn}

Equivalently, one can define the {\em non-splitting} points, using the following well-known lemma. For completeness, we supply a proof.

\begin{lem}\label{15.1.17} Let $X$ be a topological space, $n\in\N^*$, and $P:X\to \Cal P_n$  a continuous map, all values of which are monic and of degree $n$. Let $\xi\in X$, let $w_1,\ldots,w_m$ be the zeros of $P(\xi)$, and let $n_j$ be the order of $w_j$ as a zero of $P(\xi)$. Then $\xi$ is \underline{not} a splitting point of the zeros of $P$ if and only if the following condition is satisfied:

If $U$ is a sufficiently small connected open neighborhood of $\xi$, then there are uniquely determined  continuous functions $\lambda_1,\ldots,\lambda_m:U\to \C$, which are holomorphic if $X$ is a complex space and $P$ is holomorphic, such that

-- $\lambda_j(\xi)=w_j$ for $1\le j\le m$,

-- for each $\zeta\in U$, $\lambda_1(\zeta),\ldots,\lambda_m(\zeta)$ are the zeros of $P(\zeta)$ \footnote{In particular, $\lambda_i(\zeta)\not=\lambda_j(\zeta)$ if $i\not=j$, according to our convention at the end of Section \ref{22.2.15}.}, and the orders of these zeros are $n_1,\ldots,n_m$, respectively.
\end{lem}
\begin{proof}
It is clear that the condition is sufficient.

Assume that $\xi$ is not a splitting point of the zeros of $P$.

Then,  by definition, there is a neighborhood $U$ of $\xi$ such that
\begin{equation}\label{1.2.17}
m\ge \text{ the numbers of zeros of }P(\zeta),\quad\text{ for all } \zeta\in U.
\end{equation}  Choose $\varepsilon>0$  such that the disks
\begin{equation}\label{1.2.17'}\D_j:=\big\{z\in \C\,\big\vert\,\vert z-w_j\vert< \varepsilon\big\}, \quad 1\le j\le m,\end{equation}
are pairwise disjoint.
Since $P$ is continuous and $P(\xi)(z)\not=0$ for $z\in (\partial \D_1\cup\ldots\cup \partial \D_m)$, shrinking $U$,  we can achieve  that $\vert P(\zeta)(z)-P(\xi)(z)\vert<\vert P(\xi)(z)\vert$, for all $\zeta\in U$ and $z\in (\partial \D_1\cup\ldots\cup \partial \D_m)$.  Then it follows from  Rouche's theorem that, for each $\zeta\in U$ and each $1\le j\le m$, \underline{counting} multiplicities, $P(\zeta)$ has exactly $n_j$ zeros in $\D_j$.
Since the disks $\D_1,\ldots,\D_m$ are pairwise disjoint and by \eqref{1.2.17}, this implies:

-- for all $\zeta\in U$ and $1\le j\le m$, $P(\zeta)$ has exactly one zero in $\D_j$,   $\lambda_j(\zeta)$, where $n_j$ is the multiplicity of this zero,

-- for all $\zeta\in U$, $\lambda_k(\zeta)\not=\lambda_j(\zeta)$ if $1\le k,j\le m$ with $k\not=j$,

-- for all $\zeta\in U$,
 $\lambda_1(\zeta),\ldots,\lambda_m(\zeta)$ are the zeros of $P(\zeta)$.

\noindent It remains to prove that the so defined functions $\lambda_1,\ldots,\lambda_m:U\to\C$ are continuous (resp. holomorphic) in $U$.

Let  $\zeta\in U$.
Since $\lambda_j(\zeta)$ is the only zero of $P(\zeta)$ in $\D_j\cup\partial\D$ and the order of this zero is $n_j$,  the function
\[z\longmapsto z\frac{P(\zeta)'(z)}{P(\zeta)(z)},\] where $P(\zeta)'$ is the complex derivative of  $P(\zeta)$, has exactly one singularity in $\D_j\cup\partial\D_j$, namely $\lambda_j(\zeta)$, and the residuum of this singularity is $n_j\lambda_j(\zeta)$. Hence
\begin{equation}\label{2.2.17}
\lambda_j(\zeta)=\frac{1}{n_j2\pi i}\int\limits_{\partial \D_j}z\frac{P(\zeta)'(z)}{P(\zeta)(z)}dz\quad\text{for}\quad 1\le j\le m.
\end{equation}  This formula shows that   $\lambda_1,\ldots,\lambda_m$ are continuous, for $P$ is continuous, and, moreover,  holomorphic if $X$ is a complex space and $P$ is holomorphic.
\end{proof}

The following theorem is contained,  e.g., in the lemma at the beginning of Chapter V, \S 7.1 of \cite{L}, applied to the  projection \[\big\{(\zeta,\lambda)\in X\times \C\,\big\vert\,P(\zeta)(\lambda)=0\big\}\longrightarrow X.\]

\begin{thm}\label{24.11.16}
 Let $X$ be a complex space and let $P:X\to \Cal P_n$ be a holomorphic map, all values of which are of degree $n$ and monic.  Then the splitting points of the zeros of $P$ form a nowhere dense closed analytic subset of $X$.
\end{thm}

\begin{rem}\label{24.11.16-}
If $X$ is smooth, there are many sources for this in the literature, see, e.g., \cite[Ch. III, Satz 6.5 and Satz 6.12]{GF}, \cite[Ch. III, Theorems 4.3 and 4.6]{FG}, \cite{B3}, \cite[S3.1]{B4}. There, the fact is used that
$P$ can be written as a finite product
\begin{equation}\label{23.11.16}
P=\omega_1^{r_1}\cdot\ldots\cdot\omega_\ell^{r_\ell},
\end{equation} where $r_i\in \N^*$, each $\omega_i$ is a monic polynomial with coefficients from $\Cal O(X)$ of positive degree\footnote{By that we mean that, for some $k_i\in\N^*$, $\omega_i$ is a holomorphic map from $X$ to $\Cal P_{k_i}$ all values of which are of degree $k_i$ and monic.},
each $\omega_i$ is prime as an element of the monoid of all monic polynomials with coefficients from $\Cal O(X)$,
and $\omega_i\not=\omega_j$ if $i\not=j$. Then it is proved  that  the discriminant of the  polynomial
$\omega_1(\zeta)\cdot\ldots\cdot \omega_\ell(\zeta)$, $\Delta$, does not identically vanish, and $\{\Delta=0\}$ is the set of splitting points of the zeros of $P$.

Note that this proof also shows that the set of splitting points of the zeros of $P$, at each point of this set, is of codimension $1$ in  $X$.
\end{rem}

In this section we give a new proof of Theorem \ref{24.11.16}, which results in a more precise result with estimates. In this proof we do not use the factorization \eqref{23.11.16} (also not for the smooth part of $X$). The main tool of our proof is the following lemma, which is known (see, e.g., \cite[\S 2, 1, VII]{KN} or \cite[Theorem 0.1]{GH}). For completeness, we give a proof.

\begin{lem}\label{20.2.15}\footnote{Formulas \eqref{14.11.16} and \eqref{14.11.16'} below show  that $\pm\det \Phi$ is the  discriminant of $p$ (see, e.g., \cite[\S 35]{vdW}). Therefore, this lemma in particular contains the well-known fact that $p$ has no multiple zeros if and only if its discriminant is different from zero.} Let $p$ be a monic complex polynomial of degree $n$, $n\in\N^*$. Denote by $\Cal P_{-1}$ the space which consists only of the zero polynomial. Let \[\Phi:\Cal P_{n-2}\oplus\Cal P_{n-1}\to\Cal P_{2n-2}\]  be the linear map defined by
\begin{equation*}
\Phi(s,q)=ps-p'q\quad\text{for}\quad (s,q)\in\Cal P_{n-2}\oplus\Cal P_{n-1},
\end{equation*} where $p'$ denotes the complex derivative of $p$. Further, let $m$ be the number of zeros of $p$ (not counting multiplicities). Then
\begin{equation}\label{12.2.15+++}
\rank \Phi=n+m-1.
\end{equation}
\end{lem}

\begin{proof}Let   $\lambda_1,\ldots,\lambda_m$ be the zeros of $p$, and  $k_j$  the order of $\lambda_j$ as a zero of $p$. Since $p$ is of degree $n$ and monic, then $k_1+\ldots+k_m=n$ and
\[
p(\lambda)=(\lambda-\lambda_1)^{k_1}\ldots(\lambda-\lambda_m)^{k_m},\quad \lambda\in\C.
\] Set $q_0(\lambda)=(\lambda-\lambda_1)\ldots(\lambda-\lambda_m)$ and
$
s_0(\lambda)=\sum_{j=1}^mk_j(\lambda-\lambda_1)\ldots\;^{}_{\widehat j}\;\ldots(\lambda-\lambda_m).
$ Then
\begin{equation}\label{12.2.15m}
ps_0=p'q_0.
\end{equation}

Next we prove that
\begin{equation}\label{12.2.15m'}
\ke \Phi=\Big\{(s_0a,q_0a)\,\Big\vert\, a\in\Cal P_{n-1-m}\Big\}.
\end{equation}

{\em Proof of }``$\supseteq$'': For $m=n$ this is trivial. Let $1\le m\le n-1$ and $a\in \Cal P_{n-1-m}$. Since $s_0$ is of degree $m-1$ and $q_0$ of degree $m$,  then $(s_0a,q_0a)\in \Cal P_{n-2}\oplus\Cal P_{n-1}$, and by \eqref{12.2.15m}, $\Phi_p(s_0a,q_0a)=(ps_0-p'q_0)a=0$.

{\em Proof of }``$\subseteq$'': Let $(s,q)\in\ke \Phi$, i.e., $s\in\Cal P_{n-2}$, $q\in \Cal P_{n-1}$ and
 \begin{equation}\label{12.2.15m''}ps=p'q.
  \end{equation}Then each $\lambda_j$ is a zero  of order $\ge k_j$ of $p'q$. Since the order of $\lambda_j$ as a zero of $p'$ is $<k_j$ (for $k_j=1$, by this we mean that $p'(\lambda_j)\not=0$), it follows that each $\lambda_j$ is a zero of $q$. Hence, $q$ is of the form
 \begin{equation}\label{12.2.15-}q=q_0a,
 \end{equation} where $a$ is some complex polynomial (possibly, $a\equiv 0$). If $a\not\equiv 0$, from \eqref{12.2.15-} it follows that $\deg a=\deg q-\deg q_0$.
Since $\deg q_0=m$ and $\deg q\le n-1$, this implies that $\deg a\le n-1-m$, i.e.,
 \begin{equation}\label{21.4.15}a\in \Cal P_{n-1-m}.
 \end{equation}If $a\equiv 0$, \eqref{21.4.15} holds  trivially.
 Moreover, by \eqref{12.2.15m}, \eqref{12.2.15-} and \eqref{12.2.15m''},
\[
ps_0a=p'q_0a=p'q=ps.
\]As $p\not\equiv 0$, this implies that  $s=s_0a$. Together with \eqref{12.2.15-} and \eqref{21.4.15} this proves that $(s,q)$ belongs to the right hand side of \eqref{12.2.15m'}.

So \eqref{12.2.15m'} is proved. Now we consider
 the linear map
 \[\begin{split}\Psi:\Cal P_{n-1-m}&\longrightarrow\Cal P_{n-2}\oplus \Cal P_{n-1}\\
 a\qquad&\longmapsto \quad(s_0a,q_0a).
 \end{split}\] Since $s_0\not\equiv 0$ and $q_0\not\equiv0$, this map  is  injective. Hence
 \[\dim\im \Psi= \dim \Cal P_{n-1-m}=n-m.
 \]As, by \eqref{12.2.15m'}, $\im \Psi=\ke \Phi$, it follows  that
 \[
 \dim \ke\Phi=\dim\im\Psi=n-m.
 \]As $\rank \Phi=2n-1-\dim\ke\Phi$, this proves \eqref{12.2.15+++}.
\end{proof}

Also we use a simple fact on the jump behavior of the rank of a continuous matrix.

\begin{defn}\label{3.11.16}
Let $X$ be a topological space, and $M:X\to \mathrm{Mat}(n\times m,\C)$  a continuous map. A point $\xi\in X$ will be called a {\bf jump point} of $\rank M$ if, for each neighborhood $U$ of $\xi$, there is a point $\zeta\in U$ such that $\rank M(\zeta)>\rank M(\xi)$.

Since the function $X\ni\zeta\mapsto \rank M(\zeta)$ is lower semicontinuous, then $\xi\in X$ is \underline{not} a jump point of $\rank M$ if and only if there is a neighborhood $U$ of $\xi$ such that the map
\[
U\ni\zeta\longmapsto \rank M(\zeta)
\]is constant.
\end{defn}

\begin{lem}\label{4.11.16}Let $X$ be an irreducible complex space,  $M:X\to\mathrm{Mat}(n\times m,\C)$ holomorphic, and
\[
r_{\mathrm{max}}:=\max_{\zeta\in X}\rank M(\zeta).
\]Let $f_1,\ldots,f_\ell$ be the minors of order $r_{\mathrm{max}}$ of $M$. Then
$\{f_1=\ldots=f_\ell=0\}$ is the set of jump points of $\rank M$.
\end{lem}
\begin{proof} Let $\xi\in X$ be given.
First assume that $\xi$ is a jump point of $\rank M$. Then, in particular, $\rank M(\xi)<r_{\mathrm{max}}$. Hence $f_1(\xi)=\ldots=f_\ell(\xi)=0$.

Now we assume that $\xi$ is not a jump point of $\rank M$. We have to prove that then $\rank M(\xi)=r_{\mathrm{max}}$. Assume to the contrary that $\rank M(\xi)<r_{\mathrm{max}}$. Since $\xi$ is not a jump point of $\rank M$, then there is a neighborhood $U$ of $\xi$ such that $\rank M(\zeta)\le \rank M(\xi)$ for all $\zeta\in U$. On the other hand, as  $M$ is continuous, there is a neighborhood $V$ of $\xi$ such that $\rank M(\zeta)\ge \rank M(\xi)$ for all $\zeta\in V$. Hence, $\rank M(\zeta)= \rank M(\xi)<r_{\mathrm{max}}$ for all $\zeta\in U\cap V$, which means that $f_1=\ldots=f_\ell=0$ on $U\cap V$.
Since $X$ is irreducible and, hence, the manifold of smooth points of $X$ is connected and everywhere dense in $X$, and since the functions $f_j$ are holomorphic, it follows that $f_1=\ldots=f_\ell=0$ on all of $X$, which is  impossible by  definition of $r_{\mathrm{max}}$.
\end{proof}

Now we are ready to prove Theorem \ref{24.11.16}. Actually, we prove  the more precise

\begin{thm}\label{10.11.16} Let $X$ be a complex space, $n\in \N^*$, and let $P:X\to \Cal P_n$ be  a holomorphic map, all values of which are monic and of degree $n$. Denote by $\mathrm{split\,}P$ the set of splitting points of the zeros of $P$.

Then
$\mathrm{split\,}P$ is a nowhere dense closed analytic subset of $X$.

Moreover,  if
$X$ is irreducible and $\mathrm{split\,}P\not=\emptyset$, and if $P_0(\zeta),\ldots,P_n(\zeta)$ are the coefficients of $P(\zeta)$, then there exist  finitely many holomorphic functions $h_1,\ldots,h_\ell:X\to\C$, each of which is a finite sum of finite products of some of the coefficients  $P_0,\ldots,P_{n-1}$, such that
\begin{equation}\label{7.12.16-}\mathrm{split\,}P=\big\{h_1=\ldots =h_\ell=0\big\},\end{equation}
and
\begin{equation}\label{4.1.17}
\vert h_j(\zeta)\vert\le (2n)^{4n}\max_{0\le \mu\le n-1}\big\vert P_\mu(\zeta)\big\vert^{2n}\quad\text{for all}\quad \zeta\in X\text{ and }1\le j\le \ell.
\end{equation}
\end{thm}

\begin{proof} If $\mathrm{split\,}P=\emptyset$, the claim of the theorem is trivial. Therfore, we may assume that $\mathrm{split\,}P\not=\emptyset$.

First, moreover assume that $X$ is irreducible.

Let $L\big(\Cal P_{n-2}\oplus\Cal P_{n-1},\Cal P_{2n-2}\big)$ be the space a linear maps from $\Cal P_{n-2}\oplus\Cal P_{n-1}$ to $\Cal P_{2n-2}$, and let
\[\Phi:X\to L\big(\Cal P_{n-2}\oplus\Cal P_{n-1},\Cal P_{2n-2}\big)
\] be the holomorphic map defined by
\begin{equation}\label{4.11.16-}\Phi(\zeta)(s,q)=P(\zeta)s-P(\zeta)'q,\qquad \zeta\in X,\quad(s,q)\in\Cal P_{n-2}\oplus\Cal P_{n-1},
\end{equation} where $P(\zeta)'$ is the complex derivative of the polynomial $P(\zeta)$.

For $\lambda\in \C$, we define
\begin{align*}&u_j(\lambda)=\begin{cases}(\lambda^j,0)\quad&\text{for}\quad j=0,\ldots,n-1,\\
(0,\lambda^{j-n})\quad&\text{for}\quad j=n,\ldots, 2n-2,
 \end{cases}\\
 &v_j(\lambda)=\,\,\lambda^j\quad\quad\qquad\quad\,\,\text{for}\quad j=0,\ldots,2n-2.\end{align*}
  Then $u_0,\ldots,u_{2n-2}$ is a basis of $\Cal P_{n-1}\oplus\Cal P_{n-2}$ and $v_0,\ldots,v_{2n-2}$ is a basis of $\Cal P_{2n-2}$. Let $M=(M_{ij})_{i,j=0}^{2n-2}$
 be the corresponding representaion matrix of $\Phi$, i.e.,
 \begin{equation*}
 \Phi(\zeta)u_j=\sum_{i=0}^{2n-2}M_{ij}(\zeta)v_i\quad\text{for}\quad \zeta\in X\text{ and } 0\le j\le 2n-2.
 \end{equation*}
 Then,  by \eqref{4.11.16-}, for  $0\le j\le n-1$, we have
 \[\big(\Phi(\zeta)u_j\big)(\lambda)=\sum_{i=0}^n P_i(\zeta)\lambda^{i+j}=\sum_{i=j}^{n+j}  P^{}_{i-j}(\zeta)v_{i}(\lambda), \] and, for $n\le j\le 2n-2$, we have
\[\big(\Phi(\zeta)u_j\big)(\lambda)=- \sum_{i=1}^n iP_i^{}(\zeta)\lambda^{i-1+j-n}=\sum_{i=j-n}^{j-1}(j-i-n-1)P^{}_{i-j+n+1}(\zeta)v_i(\lambda).
\]
Hence, for $0\le j\le n-1$, we have
\begin{equation}\label{14.11.16}
M_{ij}(\zeta)=\begin{cases}P_{i-j}(\zeta)&\text{if}\quad j\le i\le n+j,\\
0&\text{otherwise},\end{cases}
\end{equation}and, for $n\le j\le 2n-2$, we have
\begin{equation}\label{14.11.16'}
M_{ij}(\zeta)=\begin{cases}(j-i-n-1)P_{i-j+n+1}(\zeta)&\text{if}\quad j-n\le i\le j-1\\
0&\text{otherwise}.\end{cases}
\end{equation}

Let \[r_{\mathrm{max}}:=\max_{\zeta\in X}\rank M(\zeta),\]
and let $h_1,\ldots,h_\ell$ be the minors of order $r_{\mathrm{max}}$ of $M$ which do not vanish identically on $X$ (by definition of $r_{\mathrm{max}}$, there are such minors). Then, by Lemma \ref{4.11.16}, $\{h_1=\ldots=h_\ell=0\}$ is the set of jump points of $\rank M$. Since, by Lemma \ref{20.2.15},
\[
\rank M(\zeta)=\text{ the number of zeros of }P(\zeta) +n-1,\quad \text{for all}\quad \zeta\in X,
\] this proves \eqref{7.12.16-}.

As
the manifold of smooth points of $X$ is a connected ($X$ is irreducible) and dense subset of $X$ and the functions   $h_j$ do not identically vanish on $X$, \eqref{7.12.16-} in particular  shows that $\mathrm{split\,}P$ is a nowhere dense closed analytic subset of $X$.

Since $\mathrm{split\,}P\not=\emptyset$, from  \eqref{7.12.16-} we moreover see that none of the functions $h_j$ is zero free. Therefore, it follows
\eqref{14.11.16} and \eqref{14.11.16'} that each $h_j$ is a finite sum of products of some of the coefficients $P_1,\ldots,P_{n-1}$ (recall that $P_n\equiv 1$) and that
\[
\vert h_j(\zeta)\vert\le r_{\mathrm{max}}^{}!\,\Big( n\max_{0\le \mu\le n-1}\big\vert P_\mu(\zeta)\vert\Big)^{ r_{\mathrm{max}}}\quad\text{for}\quad 1\le j\le \ell,
\]
 which implies
\eqref{4.1.17}.

Now we consider the general case. By the global decomposition theorem for complex spaces (see, e.g., \cite[V.4.6]{L} or \cite[Ch. 9, \S 2.2]{GR}), there is a locally finite covering $\{X_i\}_{i\in I}$ of $X$ such that each $X_i$ is an irreducible closed analytic subset of $X$. Then, clearly,
\begin{equation*}
\mathrm{split\,}P=\bigcup_{i\in I}\Big(X_i\cap \mathrm{split\,}(P\vert_{X_i})\Big),
\end{equation*}
and, as already proved, each $X_i\cap \mathrm{split\,}(P\vert_{X_i})$ is a nowhere dense analytic subset of $X_i$. Since the covering  $\{X_i\}_{i\in I}$ is locally finite, this proves that $\mathrm{split\,}P$ is a nowhere dense analytic subset of $X$.
\end{proof}

\begin{rem}\label{24.11.16--}Our  proof of Theorem \ref{10.11.16} does not show that, at each point of $\mathrm{split\,}P$ which is a smooth point of $X$, $\mathrm{split\,}P$ is of codimension 1 in $X$ (in distinction to the well-known  proof outlined in Remark \ref{24.11.16-}). An advantage of this proof is that it shows that, in the irreducible case,   $\mathrm{split\,} P$ can be defined by finite sums of finite products of the coefficients of $P$, which satisfy estimate \eqref{4.1.17}. This implies, for example:

 -- If  $X=\{\zeta\in\C\,\vert\,\vert\zeta\vert<1\}$ and the coefficients of $P$ are bounded, then  $\mathrm{split\,} P$ satisfies the Blaschke condition.

-- If $X=\C^N$, and the coefficients of $P$ are holomorphic polynomials, then $\mathrm{split\,} P$ is defined by finitely many holomorphic polynomials. For $N=1$ this means that $\mathrm{split\,}P$ is finite (which is well-known from the theory of algebraic functions).

\end{rem}

\section{Splitting points of the eigenvalues of a matrix function}\label{8.1.17}

\begin{defn}\label{24.11.16+}
Let $X$ be a topological space, and $A:X\to \mathrm{Mat}(n\times n;\C)$ continuous.
A point $\xi\in X$ is called a {\bf splitting point of the eigenvalues of $A$} if,  for each neighborhood $U$ of $\xi$, there exists $\zeta\in U$ such that  $A(\zeta)$ has more eigenvalues than $A(\xi)$ (not counting multiplicities).
\end{defn}

Since, for each  $\Phi\in \mathrm{Mat}(n\times n,\C)$, the eigenvalues of  $\Phi$ are the zeros of the  characteristic  polynomial $\det\big(\lambda -\Phi\big)$, $\lambda\in\C$, which is of degree $n$ and monic,  from  Lemma \ref{15.1.17} we immediately obtain the following
characterization of the {\em non-splitting} points of the eigenvalues of a matrix.

\begin{lem}\label{15.1.17'} Let $X$ be a topological space, $n\in\N^*$, and $A:X\to \mathrm{Mat}(n\times n,\C)$  a continuous map. Let $\xi\in X$, let $w_1,\ldots,w_m$ be the eigenvalues of $A(\xi)$, and let $n_j$ be the algebraic multiplicity of $w_j$.\footnote{i.e., the order as a zero of the characteristic polynomial of $A(\xi)$.} Then $\xi$ is \underline{not} a splitting point of the eigenvalues of $A$ if and only if the following condition is satisfied:

If $U$ is a sufficiently small connected open neighborhood of $\xi$, then there are uniquely determined continuous functions $\lambda_1,\ldots,\lambda_m:U\to \C$, which are holomorphic if $X$ is a complex space and $A$ is holomorphic, such that

-- $\lambda_j(\xi)=w_j$ for $1\le j\le m$,

-- for each $\zeta\in U$, $\lambda_1(\zeta),\ldots,\lambda_m(\zeta)$ are the eigenvalues of $A(\zeta)$, where $n_j$ is the algebraic multiplicity $\lambda_j(\zeta)$.
\end{lem}

\begin{thm}\label{25.11.16} Let $X$ be a complex space and  $A:X\to \mathrm{Mat}(n\times n,\C)$ holomorphic. Denote by $\mathrm{split\,}A$ the set of splitting points of the eigenvalues of $A$.

Then $\mathrm{split\,}A$ is a nowhere dense closed analytic subset of $X$.

 Moreover, if $X$ is irreducible and $\mathrm{split\,}A\not=\emptyset$, then there exist  finitely many holomorphic functions  $h_1,\ldots,h_\ell:X\to \C$, each of which is   a finite sum of finite products of elements of $A$,   such that
\begin{equation}\label{10.1.17} \mathrm{split\,}A=\big\{h_1=\ldots= h_\ell=0\big\},
\end{equation}
and
\begin{equation}\label{10.1.17'}
\vert h_j(\zeta)\vert\le (2n)^{6n^2} \Vert A(\zeta)\Vert^{2n^2}\quad\text{for all}\quad \zeta\in X\text{ and }1\le j\le \ell.
\end{equation}
\end{thm}
\begin{proof} Let $P(\zeta)(\lambda):=\det(\lambda-A(\zeta))$, for $\zeta\in X$ and $\lambda\in \C $, and let $\mathrm{split\,}P$ be the set of splitting points of the zeros of $P$.
Since the eigenvalues of $A$ are the zeros of $P$, then
\[
\mathrm{split\,}A=\mathrm{split\,}P.
\]Therefore, by Theorem \ref{10.11.16}, $\mathrm{split\,}A$ is a nowhere dense analytic subset of $X$.

Now we assume that $X$ is irreducible and $\mathrm{split\,}A\not=\emptyset$. Let  $P_1(\zeta),\ldots,P_n(\zeta)$ be
the coefficients of $P(\zeta)$. Then, again by Theorem \ref{10.11.16}, there exist  finitely many holomorphic functions $h_1,\ldots,h_\ell:Y\to\C$, each of which is a finite sum of finite products of some of the coefficients  $P_0,\ldots,P_{n-1}$ and, hence, a finite sum of finite products of elements of $A$, such that
\begin{equation}\label{14.3.17}\mathrm{split\,}P=\big\{h_1=\ldots =h_\ell=0\big\}
\end{equation}and
\begin{equation}\label{14.3.17'}
\vert h_j(\zeta)\vert\le (2n)^{4n}\max_{0\le \mu\le n-1}\big\vert P_\mu(\zeta)\big\vert^{2n}\quad\text{for all}\quad \zeta\in X\text{ and }1\le j\le \ell.
\end{equation}

Since $\mathrm{split\,}A=\mathrm{split\,}P$, then \eqref{10.1.17} follows from \eqref{14.3.17}.

If $0\le \mu\le n-1$, then $\big\vert P_\mu(\zeta)\big\vert\le n!\,\Vert A(\zeta)\Vert^n$ for all $\zeta\in X$. Therefore
it follows from \eqref{14.3.17'} that
\[
\vert h_j(\zeta)\vert\le  (2n)^{4n}\Big(n!\,\Vert A(\zeta)\Vert^n\Big)^{2n}\quad\text{for all}\quad \zeta\in X\text{ and }1\le j\le \ell,
\]
which  implies \eqref{10.1.17'}.
\end{proof}
\begin{rem}\label{10.1.17-} According to the end of Remark \ref{24.11.16-}, the claim of Theroem \ref{25.11.16} can be completed by the statement that, at  each point of $\mathrm{split\,}A$ which is a  smooth point of $X$,  $\mathrm{split\,}A$ is of codimension 1 in $X$.
\end{rem}

\section{Jordan stable points}\label{7.1.17'}

\begin{defn}\label{24.12.16}As usual, by a {\bf Jordan block}  we mean a matrix of the form $\lambda I_\ell +(\delta_{i,j-1})_{i,j=1}^\ell$, where $\delta_{ij}$ is the Kronecker symbol, $\lambda\in\C$  (the eigenvalue of the Jordan block) and $\ell\in\N^*$ (the size of the Jordan block).

If $\Phi\in\mathrm{Mat}(n\times n,\C)$ and  $\lambda_1,\ldots,\lambda_m$ are the eigenvalues of $\Phi$, then, for  $\ell\in\N^*$, we denote by $\vartheta_\ell\big(\Phi,\lambda_j)$ the number of Jordan blocks of size $\ell$ of the  eigenvalue $\lambda_j$   in the Jordan normal forms of $\Phi$, and  set
\[
\vartheta_\ell\big(\Phi,\bullet)=\sum_{j=1}^m\vartheta_\ell\big(\Phi,\lambda_j).
\]
Further, then we define
\[\Theta_\Phi=
\big(\lambda_1-\Phi\big)\cdot\ldots\cdot\big(\lambda_m-\Phi\big),
\]which is correct,  for the matrices
$\lambda_1-\Phi,\ldots,\lambda_m-\Phi$ pairwise commute.
\end{defn}

\begin{lem}\label{25.12.16}Let $\Phi\in \mathrm{Mat}(n\times n,\C)$,  $\lambda_1,\ldots,\lambda_m$ the eigenvalues of $\Phi$, and  $n_j$  the algebraic multiplicty of $\lambda_j$ (i.e., its order  as a zero of the characteristic polynomial of $\Phi$). Then
\begin{align}&\label{2.1.17}\rank(\lambda_j-\Phi)^k=n-n_j\quad\text{for}\quad k\ge n_j\quad\text{and}\quad 1\le j\le m\\
&\label{29.12.16}
\Theta_\Phi^k=0\quad\text{for}\quad k\ge n,
\end{align}
\begin{multline}\label{25.12.16'}
\rank \Theta_\Phi^k=n-nm+\rank(\lambda_1-\Phi)^k+\ldots+\rank(\lambda_m-\Phi)^k\\
\text{if}\quad 1\le k\le n-1,
\end{multline}
\begin{equation}\label{25.12.16n}
\rank\Theta_\Phi^k=n-\sum_{\ell=1}^{k}\ell
\vartheta_\ell\big(\Phi,\bullet)-k\sum_{\ell=k+1}^{n}\vartheta_\ell\big(\Phi,\bullet)\quad\text{if}\quad 1\le k\le n-1,
\end{equation}
\begin{multline}
\label{25.12.16'''}
\vartheta_k\big(\Phi,\lambda_j)=\rank\,(\lambda_j-\Phi)^{k-1}+\rank(\lambda_j-\Phi)^{k+1}-2\rank\,(\lambda_j-\Phi)^{k}\\\text{if}\quad 1\le  k\le n\text{ and } 1\le j\le m,
\end{multline}where $(\lambda_j-\Phi)^0:=I_n$.
\end{lem}
For completeness, we give a proof of this lemma, although
the relations collected there (and in its proof) are well-known, possibly, in somewhat different formulations, see, e.g.,  \cite[Kap. II, \S 8.4]{B2} or \cite[2.9.4]{B4}.
\begin{proof}
First note that, if, for some $1\le j\le m$, $J$ is a Jordan block of size $\ell$ and with eigenvalue $\lambda_j$, then
\begin{equation}\label{26.11.16}\begin{split}&\rank \,(\lambda_j-J)^k=\ell-k\quad\text{for}\quad 0\le k\le \ell-1,\\
&(\lambda_j-J)^\ell=0,\\
&\lambda_i-J\in \mathrm{GL}(\ell,\C)\quad\text{for all}\quad 1\le i\le m\text{ with }i\not=j.
\end{split}\end{equation}
Denote by $E_j$  the algebraic eigenspace of $\lambda_j$, i.e., $E_j:=\ke(\lambda_j-\Phi)^{n_j}$.
Then each $E_j$ is an invariant subspace of each $\lambda_i-\Phi$, and, since $\Phi$ is similar to a matrix in Jordan normal form,  it  follows from \eqref{26.11.16} that
\begin{equation}\label{29.12.16'''}
\C^n=E_1\oplus\ldots\oplus E_m,\quad\text{and}\quad n_j=\dim E_j\quad\text{for}\quad 1\le j\le m,
\end{equation}$\lambda_i-\Phi$ maps $E_j$ isomorphically onto itself if $i\not=j$,
\begin{equation}\label{29.12.16'}
\ke(\lambda_j-\Phi)^k=E_j\quad\text{for}\quad k\ge n_j\quad\text{and}\quad1\le j\le m,
\end{equation}
\begin{multline}\label{25.12.16**}
\dim\ke(\lambda_j-\Phi)^k=\sum_{\ell=1}^{k}\ell
\vartheta_\ell\big(\Phi,\lambda_j)+k\sum_{\ell=k+1}^{n_j}\vartheta_\ell\big(\Phi,\lambda_j)\\\text{for }1\le j\le m\text{ and } k\in \N^*,
\end{multline} and (taking into account that the matrices
$\lambda_j-\Phi$ pairwise commute), for all $k\in\N^*$,
\begin{align}&\label{29.12.16''}
\ke \Theta_\Phi^k=\ke(\lambda_1-\Phi)^k\oplus\ldots\oplus \ke(\lambda_m-\Phi)^k,\\
&\label{25.12.16*}
\dim\ke \Theta_\Phi^k=\dim\ke(\lambda_1-\Phi)^k+\ldots+ \dim\ke(\lambda_m-\Phi)^k.
\end{align}
From \eqref{25.12.16**} and \eqref{25.12.16*} together, we obtain
\begin{equation}\label{25.12.16***}
\dim\ke \Theta_\Phi^k=\sum_{\ell=1}^{k}\ell
\vartheta_\ell\big(\Phi,\bullet)+k\sum_{\ell=k+1}^{n}\vartheta_\ell\big(\Phi,\bullet),\quad k\in\N^*.
\end{equation}

Now: \eqref{2.1.17} follows from \eqref{29.12.16'''} and \eqref{29.12.16'};
\eqref{29.12.16} follows from \eqref{29.12.16'''}, \eqref{29.12.16'} and \eqref{29.12.16''};
\eqref{25.12.16'} follows from  \eqref{25.12.16*}; \eqref{25.12.16n} follows from \eqref{25.12.16***}.

To prove \eqref{25.12.16'''}, we first note that \eqref{25.12.16**} holds also for  $k=0$ -- then both sides are zero. Hence, for  $k\in\N^*$ and $1\le j\le m$,
\begin{equation*}\begin{split}
\dim&\ke (\lambda_j-\Phi)^k-\dim\ke(\lambda_j-\Phi)^{k-1}\\
&=\bigg(\sum_{\ell=1}^{k}\ell
\vartheta_\ell\big(\Phi,\lambda_j)-\sum_{\ell=1}^{k-1}\ell
\vartheta_\ell\big(\Phi,\lambda_j)\bigg)+\bigg(k\sum_{\ell=k+1}^{n_j}\vartheta_\ell\big(\Phi,\lambda_j)-(k-1)
\sum_{\ell=k}^{n_j}\vartheta_\ell\big(\Phi,\lambda_j)\bigg)\\
&=k\vartheta_k\big(\Phi,\lambda_j)+\sum_{\ell= k+1}^{n_j}\vartheta_\ell\big(\Phi,\lambda_j)-(k-1)\vartheta_k(\Phi,\lambda_j)=\sum_{\ell= k}^{n_j}\vartheta_\ell\big(\Phi,\lambda_j)
\end{split}\end{equation*} and, therefore,
\[
\vartheta_k\big(\Phi,\lambda_j)=2\dim\ke (\lambda_j-\Phi)^k-\dim\ke(\lambda_j-\Phi)^{k-1}-\dim\ke (\lambda_j-\Phi)^{k+1},
\]
 which implies \eqref{25.12.16'''}.
\end{proof}

\begin{lem}\label{24.12.16''}  Let $X$ be a topological space,  $A:X\to \mathrm{Mat}(n\times n,\C)$ a continuous map, and $\xi$ a point in $X$ which is \underline{not} a splitting point of the eigenvalues of $A$, and let $w_1,\ldots,w_m$ be the eigenvalues of $A(\xi)$, with the  algebraic multiplicities $n_1,\ldots,n_m$, respectively. Take a sufficiently small connected open neighborhood $U$ of $\xi$ such that (by Lemma \ref{15.1.17'}) there are uniquely determined continuous functions $\lambda_1,\ldots,\lambda_m:U\to \C$, which are holomorphic if $X$ is a complex space and $A$ is holomorphic, such that

-- $\lambda_j(\xi)=w_j$ for $1\le j\le m$, and

-- for each $\zeta\in U$, $\lambda_1(\zeta),\ldots,\lambda_m(\zeta)$ are the eigenvalues of $A(\zeta)$, where $n_j$ is the  algebraic multiplicity of $\lambda_j(\zeta)$.

Then  the following conditions are equivalent.

{\em (i)}  There exists a neighborhood $V\subseteq U$ of $\xi$ such that the map
\begin{equation}\label{25.12.16+}
V\ni\zeta\longmapsto \vartheta_\ell\big(A(\zeta),\lambda_j(\zeta)\big)
\end{equation}is constant if  $1\le j\le m$ and $1\le\ell\le n$.

{\em (ii)} There exists a neighborhood $V\subseteq U$ of $\xi$ such that the map
\begin{equation}\label{25.12.16+neu}
V\ni\zeta\longmapsto \vartheta_\ell\big(A(\zeta),\bullet\big)
\end{equation}is constant if $1\le\ell\le n$.

{\em (iii)} There exists a neighborhood $V\subseteq U$ of $\xi$ such that the map
\begin{equation}\label{24.12.16-}
V\ni\zeta\longmapsto \rank \big(\Theta_{A(\zeta)}\big)^k
\end{equation}is constant if $1\le k\le n-1$.

{\em (iv)} There exists a neighborhood $V\subseteq U$ of $\xi$ such that the map
\begin{equation}\label{24.12.16-neu}
V\ni\zeta\longmapsto \rank \big(\lambda_j(\zeta)-A(\zeta)\big)^k
\end{equation}is constant if $1\le k\le n-1$.

{\em (v)} There exists a neighborhood $V\subseteq U$ of $\xi$ and a continuous map $T:V\to \mathrm{GL}(n,\C)$, which is holomorphic if $X$ is a complex space and $A$ is holomorphic, such that $T(\zeta)^{-1}A(\zeta)T(\zeta)$ is in Jordan normal form for all $\zeta\in V$.
\end{lem}
If $X$ is a domain in $\C$ and $A$ is holomorphic,  the equivalence of conditions (i), (ii) and (v)  is due to  G. P. A. Thiesse \cite{T}.

\begin{proof} The equivalence of (i) - (iv) follows from Lemma \ref{25.12.16}. Indeed:

(i) $\Rightarrow$ (ii) is trivial.

(ii) $\Rightarrow$ (iii) follows from \eqref{25.12.16n}.

To prove that (iii) $\Rightarrow$ (iv), we note that, by \eqref{25.12.16'},
\[\rank \big(\Theta_{A(\zeta)}\big)^k=\mathrm{const}+\rank\big(\lambda_1(\zeta)-A(\zeta)\big)^k+\ldots+\rank\big(\lambda_m(\zeta)-A(\zeta)\big)^k
\]for all $\zeta\in U$ and $1\le k\le n-1$, and observe that the functions on the right hand side of this relation are lower semicontinuous in $\zeta$ (since the rank of a continuous matrix function is always lower semicontinuous). If  the left hand side is constant for $\zeta$  in some neighborhood $V$ of $\xi$, this is possible only if also  the functions on the right hand side are constant for $\zeta\in V$.

(iv) $\Rightarrow$ (i) follows from   \eqref{25.12.16'''}.

Moreover, it is clear that  (v) $\Rightarrow$ (i). To complete the proof of the lemma, therefore is is sufficient to prove that (i) $\Rightarrow$ (v).

Assume (i) is satisfied.

For each $1\le j\le m$, choose a matrix $N_j\in \mathrm{Mat}(n_j\times n_j,\C)$, which is a block diagonal matrix with  Jordan blocks on the diagonal, each of which has the eigenvalue $0$, and such that, for each $\ell\in\N^*$, exactly $\vartheta_\ell\big(A(\xi),\lambda_j(\xi)\big)$ of them are of size $\ell$. Since the algebraic multiplicity of $\lambda_j(\xi)$ is $n_j$, then $N_j$ is and $n_j\times n_j$ matrix. Hence, for each $1\le j\le m$ and each $\zeta\in V$,   $\lambda_j(\zeta)I_{n_j}+N_j$ is a Jordan block with eigenvalue $\lambda_j(\zeta)$ and of size $n_j$. Let $J:U\to\mathrm{Mat}(n\times n,\C)$ (note that $n_1+\ldots+n_m=n$) be the map such that $J(\zeta)$, $\zeta\in U$, is the block diagonal matrix with the diagonal
\[
\lambda_1(\zeta)I_{n_1}+N_1,\ldots,\lambda_m(\zeta)I_{n_m}+N_m.
\] Since the functions $\lambda_j$ are continuous, and holomorphic if $A$ is holomorphic, then
$J$ is continuous, and holomorphic  if $A$ is holomorphic. Moreover, since condition (i) is satisfied, there is a neighborhood $V\subseteq U$ of $\xi$ such that, \[\vartheta_\ell\big(A(\zeta),\lambda_j(\zeta)\big)=\vartheta_\ell\big(A(\xi),\lambda_j(\xi)\big)\quad \text{for all}\quad \zeta\in V,\]
which means that, for each $\zeta\in V$, $J(\zeta)$ is a Jordan normal form of $A(\zeta)$, i.e.,   there exists a matrix $\Theta_\zeta\in \mathrm{GL}(n,\C)$ with
\begin{equation}\label{5.12.16}\Theta^{}_\zeta J(\zeta)\Theta^{-1}_\zeta =A(\zeta).
\end{equation}

Now  let $\End\big(\mathrm{Mat}(n\times n,\C)\big)$ be the space of linear endomorphisms of the complex vector space $\mathrm{Mat}(n\times n,\C)$. Following an idea of W. Wasow \cite{W},  we consider the continuous (and holomorphic if $A$ is holomorphic) map  $\varphi:V\to \End\big(\mathrm{Mat}(n\times n,\C)\big)$ defined by
\[
\varphi(\zeta)\Phi= \Phi A(\zeta)-J(\zeta)\Phi,\qquad \zeta\in V,\quad \Phi\in \mathrm{Mat}(n\times n,\C).
\]
We claim that the map
\begin{equation}\label{27.12.16''}
V\ni\zeta\longmapsto \dim\ke\varphi(\zeta)
\end{equation}is constant.
Indeed, by definition of $\varphi$ and by \eqref{5.12.16}, for all $\zeta\in V$,
\[\begin{split}
\ke \varphi(\zeta)&=\Big\{\Phi\in\mathrm{Mat}(n\times n,\C)\;\Big\vert\; \Phi A(\zeta)=J(\zeta)\Phi\Big\}\\
&=\Big\{\Phi\in\mathrm{Mat}(n\times n,\C)\;\Big\vert\; \Phi\Theta^{}_\zeta J(\zeta)=J(\zeta)\Phi \Theta_\zeta^{}\Big\}\\
&=\Big\{\Phi\in\mathrm{Mat}(n\times n,\C)\;\Big\vert\; \Phi J(\zeta)=J(\zeta)\Phi\Big\}\Theta_\zeta^{-1}.\end{split}\]
In particular, for all $\zeta\in V$,
 \begin{equation}\label{27.12.16}\dim\ke \varphi(\zeta)=\dim\Big\{\Phi\in \mathrm{Mat}(n\times n,\C)\;\Big\vert\;\Phi J(\zeta)=J(\zeta)\Phi\Big\}.
 \end{equation}
  Since $\lambda_i(\zeta)\not=\lambda_j(\zeta)$ if $i\not=j$,  it follows from \cite[Ch. VIII, \S 1]{Ga} that, for all $\zeta\in W$, a matrix belongs to the space on the right hand side of \eqref{27.12.16} if and only if it is a
 block diagonal matrix with a diagonal of the form $\Lambda_1,\ldots,\Lambda_m$, where $\Lambda_j$ belongs to the space
 \begin{multline*}
 \Big\{\Phi\in \mathrm{Mat}(n_j\times n_j,\C)\;\Big\vert\;\Phi \big(\lambda_j(\zeta)+N_j\big)=\big(\lambda_j(\zeta)+N_j\big)\Phi\Big\}\\
 =\Big\{\Phi\in \mathrm{Mat}(n_j\times n_j,\C)\;\Big\vert\;\Phi N_j=N_j\Phi\Big\}.
 \end{multline*} Since the latter space
  is independent of $\zeta$, this means that \eqref{27.12.16''} is constant.

Since $\varphi$ is continuous, and holomorphic if  $A$ is holomorphic, the constancy of \eqref{27.12.16''} means that the family $\{\ke\varphi(\zeta)\}_{\zeta\in V}$ is a sub-vector bundle of the product bundle $W\times \mathrm{Mat}(n\times n,\C)$, which is  holomorphic if $A$
is holomorphic (see, e.g., \cite[Lemma 1]{W} or  \cite[Corollary 2]{Sh}).

Therefore, through each point in this sub-vector bundle goes a local continuous (resp. holomorphic) section. Since, by \eqref{5.12.16}, $(\xi,\Theta_\xi^{-1})$ is such a point, it follows that there is a neighborhood $V$ of $\xi$ and a continuous (resp. holomorphic) map $S:V\to\mathrm{Mat}(n\times n,\C)$ with
$S(\xi)=\Theta_\xi^{-1}$ and $S(\zeta)A(\zeta)=J(\zeta)S(\zeta)$ for all $\zeta\in V$. Since $\Theta_\xi^{-1}$ is invertible,  shrinking $V$, we may achieve that moreover $S(\zeta)\in \mathrm{GL}(n,\C)$ for all $\zeta\in V$. It remains to set $T(\zeta)=S(\zeta)^{-1}$ for $\zeta\in V$.
\end{proof}
\begin{defn}\label{6.2.17}
Let $X$ be a topological space, and $A:X\to \mathrm{Mat}(n\times n,\C)$ a continuous map. A point $\xi\in X$ is called {\bf Jordan stable} for $A$ if $\xi$ is not a splitting point of the eigenvalues of $A$ and  and the equivalent conditions (i) - (v) in Lemma \ref{24.12.16''} are satisfied.
\end{defn}

If $G$ is a domain in some $\C^N$ and $A:G\to\mathrm{Mat}(n\times n,\C)$ is holomorphic, H. Baumg\"artel  proved that there exists a nowhere dense analytic subset $B$ of $G$, which contains the splitting points of $A$,  such that all points of $G\setminus B$ are Jordan stable for $A$ (he proved that condition (v) in Lemma \ref{24.12.16''} is satisfied), see \cite{B1}, \cite[Kap. V, \S 7]{B2} and \cite[5.7]{B4} if $N=1$, and \cite{B3} and \cite[S 3.4]{B4} for arbitrary $N$.

\vspace{2mm}

In the present section we give a new proof of Baumg\"artel's theorem, which gives the following  more precise and more general

\begin{thm}\label{1.12.16'} Let $X$ be a complex space, and $A:X\to \mathrm{Mat}(n\times n,\C)$ holomorphic. Denote by $\mathrm{Jst\,}A$ the set of Jordan stable points of $A$.

Then $X\setminus \mathrm{Jst\,}A$ is a nowhere dense closed analytic subset of $X$.

Moreover, if $X$ is irreducible and normal\footnote{For the definition of a {\em normal} complex space, see, e.g., \cite[Ch. VI, \S 2]{L}. For example, each complex manifold is normal}, and if $\mathrm{Jst\,}A\not=X$, then there exist  finitely many holomorphic functions
$h_1,\ldots,h_\ell:X\to \C$, each of which is  a finite sum of finite products of  elements of $A$,   such that
\begin{equation}\label{6.12.16-}X\setminus\mathrm{Jst\,}A=\big\{h_1=\ldots= h_\ell=0\big\}\end{equation}
and
\begin{equation}\label{6.12.16--}
\vert h_j(\zeta)\vert\le (2n)^{2n^4}\Vert A(\zeta)\Vert^{2n^4}\quad\text{for all}\quad \zeta\in X\text{ and }1\le j\le \ell.
\end{equation}
\end{thm}
\begin{proof} For $\mathrm{Jst\,}A=\emptyset$, the claim of the theorem is trivial. Therefore we may assume that $\mathrm{Jst\,}A\not=\emptyset$.

 We first consider the case when $X$ is normal and irreducible.

Let $\mathrm{split\,}A$ be the set of splitting points of the eigenvalues of $A$, and let $X^0$ be the manifold of smooth points of $X$. Since $X^0$ is connected ($X$ is irreducible) and dense in $X$,  and $\mathrm{split\,}A$ is a nowhere dense analytic subset of $X$ (Theorem \ref{25.11.16}),  $X\setminus \mathrm{split\,}A$ is connected.

Consider the map
\begin{equation}\label{28.12.16}
X\setminus \mathrm{split\,}A\ni\zeta\longmapsto \Theta_{A(\zeta)}.
\end{equation}

By Lemma \ref{15.1.17'}, for each $\xi\in X\setminus \mathrm{split\,}A$, we have an open neighborhood $U_\xi\subseteq  X\setminus \mathrm{split\,}A$ of $\xi$ and holomorphic functions $\lambda^{(\xi)}_1,\ldots,\lambda^{(\xi)}_m:U_\xi\to \C$ such that, for all $\zeta\in U_\xi$, $\lambda^{(\xi)}_1(\zeta),\ldots,\lambda^{(\xi)}_m(\zeta)$ are the eigenvalues of $A(\zeta)$ and, hence,
\begin{equation}\label{28.12.16''}
\Theta_{A(\zeta)}=\big(\lambda_1^{(\xi)}(\zeta)-A(\zeta)\big)\cdot\ldots\cdot\big(\lambda^{(\xi)}_m(\zeta)-A(\zeta)\big).
\end{equation}
In particular this
 shows that \eqref{28.12.16} is holomorphic on $X\setminus \mathrm{split\,}A$.

Moreover, as $\vert\lambda_j(\zeta)\vert\le \Vert A(\zeta)\Vert$, from \eqref{28.12.16''} it follows that
\begin{equation}\label{28.12.16'}
\big\Vert \Theta_{A(\zeta)}\big\Vert\le 2^m\Vert A(\zeta)\Vert^m\quad\text{for all}\quad \zeta\in X\setminus \mathrm{split\,}A.
\end{equation} Since $X\cap \mathrm{split\,}A$ is a nowhere dense analytic subset of $X$, and $X$ is normal, this implies that \eqref{28.12.16} extends holomorphically to $X$. We denote this extended map by $\Theta$. By \eqref{28.12.16'}, then
\begin{equation}\label{3.3.17}
\big\Vert \Theta(\zeta)^k\big\Vert\le 2^{mk}\Vert A(\zeta)\Vert^{mk}\quad\text{for all}\quad \zeta\in X\text{ and }1\le k\le n.
\end{equation}Set
\[r_k=\max_{\zeta\in X}\rank\Theta(\zeta)^k\quad\text{for}\quad 1\le k\le n.
\]

{\em First case:} $r_1=0$. Then  $\big(\Theta_{A(\zeta)}\big)^k=0$ for all $\zeta\in X\setminus\mathrm{split\,}A$ and $k\in \N^*$.  In particular,  each  $\xi\in X\setminus\mathrm{split\,}A$ satisfies condition (iii) in Lemma \ref{24.12.16''}. Hence $X\setminus\mathrm{Jst\,}A=\mathrm{split\,}A$,
and the claim of the theorem  follows from Theorem \ref{25.11.16}.

{\em Second case:} $r_1>0$.
Then, by \eqref{29.12.16}, $n\ge 2$ and there is an integer $1\le k_0\le n-1$ with $r^{}_{k_0}>0$ and  $r^{}_{k_0+1}=0$. For $1\le k\le k_0$, let $f_1^{(k)},\ldots,f_{s_k}^{(k)}$ be the  minors of order $r_k$ of $\Theta^k$ which do not vanish identically on $X$.
Since $X$ is irreducible (i.e., the manifold of smooth points of $X$ is connected), and the functions $f_j^{(k)}$ are holomorphic and $\not\equiv 0$, none of them can vanish identically on an open subset of $X$. Hence,
\begin{equation}\label{5.3.17}
Z:=\bigcup_{k=1}^{k_0}\big\{f_1^{(k)}=\ldots=f_{s_k}^{(k)}=0\big\}
\end{equation}is a nowhere dense  analytic subset of $X$, and $\xi\in Z$ if and only if $\xi$  is a jump point (Def. \ref{3.11.16}) for at least one of the maps $\Theta^1,\ldots,\Theta^{k_0}$. Since $\Theta^k\equiv 0$ if $k_0+1\le k\le n-1$, the latter means that $\xi\in Z$ if and only if $\xi$ is a jump point for at least one of the maps $\Theta^1,\ldots,\Theta^{n-1}$.
In particular, $\xi\in Z\cap (X\setminus\mathrm{split\,}A)$ if and only if  $\xi\in(X\setminus\mathrm{split\,})A$ and $\xi$ is a jump point of at least one of the maps
\[
X\setminus\mathrm{split\,}A\longmapsto \big(\Theta_{A(\zeta)}\big)^1,\quad\ldots \quad,X\setminus\mathrm{split\,}A\longmapsto \big(\Theta_{A(\zeta)}\big)^{n-1},
\]i.e.,  $\xi\in Z\cap (X\setminus\mathrm{split\,}A)$ if and only if $\xi\in X\setminus\mathrm{split\,}A$ and $\xi$ violates condition (iii) in  Lemma \ref{24.12.16''}. Hence
\[
(X\setminus \mathrm{Jst\,}A)\cap (X\setminus \mathrm{split\,}A)= Z\cap (X\setminus \mathrm{split\,}A).
\]
Since $\mathrm{split\,}A\subseteq X\setminus\mathrm{Jst\,}A$, it follows that
\begin{equation}\label{3.1.17}
X\setminus\mathrm{Jst\,}A=Z\cup \mathrm{split\,}A.
\end{equation}

By  Theorem \ref{25.11.16}, we have finitely many holomorphic functions  $g_1,\ldots,g_p:X\to \C$, each of which
is a finite sum of finite products of elements of $A$, such that
\begin{equation}\label{3.1.17'}\mathrm{split\,}A=\big\{g_1=\ldots =g_p=0\big\},
\end{equation}and
\begin{equation}\label{3.1.17''}
\vert g_j(\zeta)\vert\le (2n)^{6n^2}\Vert A(\zeta)\Vert^{2n^2}\quad\text{for all }\zeta\in X\text{ and }1\le j\le p.
\end{equation}

Now let $\{h_1,\ldots,h_\ell\}$ be the set of all functions of the form
\[
g_j^{}\cdot\prod_{k=1}^{k_0} f^{(k)}_{\kappa_{k}}
\]
with $1\le j\le p$ and $1\le \kappa_k\le s_k$ for $1\le k\le k_0$. Then \eqref{6.12.16-} follows from \eqref{5.3.17}, \eqref{3.1.17}  and \eqref{3.1.17'}.

By \eqref{3.3.17}, for all $1\le k\le k_0$, $1\le j\le s_k$ and  $\zeta\in X$, we have
 \begin{equation*}
 \big\vert f_{j}^{(k)}(\zeta)\big\vert\le r_k!2^{mkr_k}\Vert A(\zeta)\Vert^{mkr_k}.
\end{equation*} Together with \eqref{3.1.17''} this yields estimate \eqref{6.12.16--} (recall that $n\ge 2$).

Next we consider the case when $X$ is irreducible, but, possibly, not normal.

Let $\pi:\widetilde X\to X$ be the normalization of $X$ (see, e.g., \cite[Ch. VI, \S 4]{L}) and $\widetilde A:=A\circ \pi$. Then $\widetilde X$ is normal and irreducible. Therefore, by part (i) of the theorem,
$\widetilde X\setminus \mathrm{Jst\,}\widetilde A$ is a nowhere dense closed analytic subset of $X$. Since, clearly,
\begin{equation}\label{5.1.17}
\pi\big(\widetilde X\setminus \mathrm{Jst\,}\widetilde A\big)=X\setminus \mathrm{Jst\,} A,
\end{equation}this implies, by Remmert's proper mapping theorem (see, e.g., \cite[Ch. V, \S 5.1]{L})), that  $X\setminus \mathrm{Jst\,}\widetilde A$ is a  closed analytic subset of $X$.

To prove that $X\setminus \mathrm{Jst\,} A$ is nowhere dense in $X$, let $X^0$ be the manifold of smooth points of $X$. Then $\pi$ is biholomorphic between $\pi^{-1}(X^0)$ and $X^0$, and, by \eqref{5.1.17},
\begin{equation*}
\pi\big(\pi^{-1}(X^0)\setminus \mathrm{Jst\,}\widetilde A\big)=X^0\setminus \mathrm{Jst\,} A.
\end{equation*} Since $\pi^{-1}(X^0)\setminus \mathrm{Jst\,}\widetilde A$ is nowhere dense in $\pi^{-1}(X^0)$, this implies that $X^0\setminus \mathrm{Jst\,} A$ is nowhere dense in $X^0$. Since $X\setminus X^0$ is nowhere dense in $X$, it follows that $X\setminus \mathrm{Jst\,}A$ is nowhere dense in $X$.

Finally, we consider the general case.

 By the global decomposition theorem for complex spaces (see, e.g., \cite[V.4.6]{L} or \cite[Ch. 9, \S 2.2]{GR}), there is a locally finite covering $\{X_i\}_{i\in I}$ of $X$ such that each $X_i$ is an irreducible closed analytic subset of $X$. Then, as already proved, each $X_i\setminus \mathrm{Jst\,}(A\vert_{X_i})$ is a nowhere dense analytic subset of $X_i$. Since the covering  $\{X_i\}_{i\in I}$ is locally finite and, clearly,
\begin{equation*}
X\setminus\mathrm{Jst\,}A=\bigcup_{i\in I}\Big(X_i\setminus \mathrm{Jst\,}(A\vert_{X_i})\Big),
\end{equation*}
this proves that $X\setminus\mathrm{Jst\,}A$ is a nowhere dense analytic subset of $X$.
\end{proof}

\begin{rem}\label{10.1.17+} Estimate \eqref{6.12.16--} shows that the claim of Theorem \ref{1.12.16'} can be completed. For example:

 -- If  $A$ is bounded, then $X\setminus\mathrm{Jst\,} A$ can be defined by bounded holomorphic functions. In the  case of the disk $X=\{\zeta\in\C\,\vert\,\vert\zeta\vert<1\}$ this implies that $X\setminus\mathrm{Jst\,} P$ satisfies the Blaschke condition.

-- If $X=\C^N$ and the elements of  $A$ are holomorphic polynomials, then $\C^N\setminus\mathrm{Jst\,} A$ is the common zero set of finitely many holomorphic polynomials, i.e., it is affine algebraic. For $N=1$ this means that $\C\setminus\mathrm{Jst\,} A$ is finite.

\end{rem}

\begin{rem}\label{10.1.17++} It is possible (in contrast to   Remark \ref{10.1.17-}) that the set of points which are not Jordan stable is of codimension $>1$, also at smooth points. Here is an example. Let
\[
A(z,w):=\begin{pmatrix} zw&-z^2\\w^2&-zw\end{pmatrix}\quad\text{for}\quad (z,w)\in \C^2.
\]
Then $A(z,w)^2=0$ for all $(z,w)\in \C^2$, and $A(z,w)=0$ if and only if $(z,w)=0$. This means that $\big(\begin{smallmatrix}0&0\\0&0\end{smallmatrix}\big)$ is the Jordan normal form of  $A(0,0)$, whereas, for all $(z,w)\in \C^2\setminus\{(0,0)\}$, $\big(\begin{smallmatrix}0&1\\0&0\end{smallmatrix}\big)$ is the Jordan normal form of $A(z,w)$. Hence, $(0,0)$ is the only point in $\C^2$ which is not Jordan stable for $A$.\end{rem}

\section{Continuous boundary values}\label{4.12.16'}

In Sections \ref{8.1.17} and \ref{7.1.17'},  we described the set of Jordan stable points for holomorphic matrices. Corresponding results can be obtained also for certain other classes of continuous matrices. In the present section we consider the following example.

Let $X$ be a connected open set of some $\C^N$ and let $\overline X$ be the closure of $X$ in $\C^N$. Assume that the boundary of $X$, $\partial X:=\overline X\setminus X$, is $\Cal C^0$ smooth.\footnote{By that we mean that, for each $\xi\in \partial X$, there is a neighborhood of $\xi$ in $\overline X$ which is homeomorphic to a set of the form $U\cap\big\{(x_1,\ldots,x_{2N})\in \R^{2N}\,\big\vert\, x_1\ge 0\big\}$, where $U$ is an open subset of $\R^{2N}$.}
Let   $A:\overline X \to \mathrm{Mat}(n\times n,\C)$ be continuous on $\overline X$, and  holomorphic in $X$. Denote by $\mathrm{split\,}A$ the set of splitting points of the eigenvalues of  $A$ (Def. \ref{24.11.16+} ), and by $\mathrm{Jst\,}A$ the set of Jordan stable points of  $A$ (Def. \ref{6.2.17} ).

Then it is easy to check that only minor  modifications of the proof of Theorem \ref{25.11.16} (and the proofs of the results used in this proof) are necessary to obtain the following

\begin{thm}\label{11.1.17} If $\mathrm{split\,}A\not=\emptyset$, then there exist  finitely many continuous on $\overline X$ and holomorphic in $X$  functions on $h_1,\ldots,h_\ell:\overline X\to \C$, each of which is a finite sum of finite products  of elements of $A$, such that
\[\mathrm{split\,}A=\big\{h_1=\ldots= h_\ell=0\big\}
\]and
\begin{equation}\label{10.1.17'neu}
\vert h_j(\zeta)\vert\le (2n)^{6n^2} \Vert A(\zeta)\Vert^{2n^2}\quad\text{for all}\quad \zeta\in \overline X\text{ and }1\le j\le \ell.
\end{equation}
\end{thm}

To describe the full set $X\setminus\mathrm{Jst\,}A$, we need some  preparation.

\begin{lem-defn}\label{12.1.17+}  Let $\xi\in\mathrm{split\,A}$,  let $\mu_1,\ldots,\mu_k$ be the eigenvalues of $A(\xi)$, and set $\D(\mu_j,\varepsilon):=\{\lambda\in \C\,\vert\, \vert \lambda-\mu_j\vert<\varepsilon\}$ for $1\le j\le k$. Then there exist positive integers $\kappa_1,\ldots,\kappa_k$ such that the following holds:

 If $\varepsilon>0$ and the disks $\D(\mu_1,\varepsilon),\ldots,\D(\mu_k,\varepsilon)$ are pairwise disjoint, then there exists a neighborhood $U_\xi$ in $\overline X$ of $\xi$  such that

{\em (i)} for each $\zeta\in U_\xi$,  all eigenvalues of $A(\zeta)$ lie
in  $D(\mu_1,\varepsilon)\cup\ldots\cup D(\mu_k,\varepsilon)$,

{\em  (ii)} for each $\zeta\in U_\xi\setminus \mathrm{split\,}A$ and each $1\le j\le k$, exactly $\kappa_j$ of the eigenvalues of $A(\zeta)$ lie in $\D(\mu_j,\varepsilon)$.

We  call $\kappa_j$ the {\em{\bf splitting amount}} of $\mu_j$.
\end{lem-defn}
\begin{proof} Let  $\varepsilon>0$ be given such that the disks $\D(\mu_1,\varepsilon),\ldots,\D(\mu_k,\varepsilon)$ are pairwise disjoint.

Let $P_A(\zeta)(\lambda):=\det(\lambda-A(\zeta))$ be the characteristic polynomial of $A(\zeta)$.
Since $A$ is continuous and no zeros of $P_A(\xi)$ lie on $ \partial \D(\mu_1,\varepsilon)\cup\ldots\cup \partial \D(\mu_k,\varepsilon)$, we can find a neighborhood $U_\xi$ in $\overline X$ of $\xi$ so small that
\begin{multline}\label{13.1.17}
\big\vert P_A(\zeta)(\lambda)-P_A(\xi)(\lambda)\big\vert<\big\vert P_A(\xi)(\lambda)\big\vert\\\text{for all }\zeta\in U_\xi\text{ and }\lambda\in \partial \D(\mu_1,\varepsilon)\cup\ldots\cup \partial \D(\mu_k,\varepsilon).
\end{multline}
 Since  $\partial X$ is $\Cal C^0$ smooth, we may moreover assume that $U_\xi\cap X$ is connected. As $\mathrm{split\,}A$ is a nowhere dense analytic subset of $X$ (Theorem \ref{25.11.16}), it follows that $(U_\xi\cap X)\setminus \mathrm{split\,}A$ is connected. Since $\overline X\setminus \mathrm{split\,}A$ is open in $\overline X$ (by Theorem \ref{11.1.17}) and since $\partial D$ is $\Cal C^0$ smooth, the connectedness of $(U_\xi\cap X)\setminus \mathrm{split\,}A$ implies that also  $U_\xi\setminus \mathrm{split\,}A$ is connected.

Let $\nu_j$ be the algebraic multiplicity of $\mu_j$ as an eigenvalue of $A(\xi)$, i.e., its order as a zero of $P_A(\xi)$. Then
\begin{equation}\label{13.1.17'}
\nu_1+\ldots+\nu_k=n.
\end{equation}
From \eqref{13.1.17} it follows, by Rouch\'e's theorem, that, for each $\zeta\in U_\xi$ and each $1\le j\le k$, \underline{counting} multiplicities, exactly $\nu_j$ zeros of $P_A(\zeta)$ lie in $\D(\mu_j,\varepsilon)$. By \eqref{13.1.17'} this implies that, for each $\zeta\in U_\xi$, counting multiplicities, exactly $n$ zeros of $P_A(\zeta)$ lie in $\D(\mu_1,\varepsilon)\cup\ldots\cup\D(\mu_k,\varepsilon)$. Since the degree of $P_A(\zeta)$ is $n$,  this proves (i).

To prove (ii), for $\zeta\in U_\xi\setminus\mathrm{split\,}A$ and $1\le j\le k$, we denote by $\kappa_j(\zeta)$ the number of eigenvalues of $A(\zeta)$ in $\D(\mu_j,\varepsilon)$, not counting multiplicities.  We have to prove that the functions
\begin{equation}\label{13.1.17-}U_\xi\setminus\mathrm{split\,}A\ni\zeta\longmapsto \kappa_j(\zeta)
\end{equation} are constant. Since $U_\xi\setminus\mathrm{split\,}A$ is connected, it is sufficient to prove that these maps are locally constant.

For that, fix $\zeta_0\in U_\xi\setminus\mathrm{split\,}A$, and let  $w^{(j)}_s$, $1\le s\le\kappa_j(\zeta_0)$, be the eigenvalues of $A(\zeta_0)$ which lie in $\D(\mu_j,\varepsilon)$,  $1\le j\le k$.
Choose $\delta>0$ so small that the disks
\[
\D\big(w^{(j)}_s,\delta\big):=\Big\{\lambda\in\C\;\Big\vert\; \big\vert\lambda-w^{(j)}_s\big\vert<\delta\Big\},\quad 1\le j\le k,\;1\le s\le \kappa_j(\zeta_0),
\]are pairwise disjoint, and
\begin{equation}\label{13.1.17'''}
\D\big(w_s^{(j)},\delta\big)\subseteq \D(\mu_j,\varepsilon)\quad
\text{for all }1\le j\le k\text{ and }1\le s\le \kappa_j(\zeta_0).
\end{equation}  Then, by Lemma \ref{15.1.17'}, there exists a neighborhood $V_{\zeta_0}\subseteq U_\xi\setminus \mathrm{split\,}A$ of $\zeta_0$ and continuous functions
\[
\lambda_{s}^{(j)}:V_{\zeta_0}\longrightarrow \D\big(w_s^{(j)},\delta\big)
\]such that, for  all $\zeta\in V_{\zeta_0}$,
\[
\lambda^{(j)}_s(\zeta),\quad 1\le j\le k,\;1\le s\le \kappa_j(\zeta_0),
\]are the eigenvalues of $A(\zeta)$. Since the disks $\D\big(w_s^{(j)},\delta\big)$ are pairwise disjoint, and by \eqref{13.1.17'''}, this in particular means that $\kappa_j(\zeta)=\kappa_j(\zeta_0)$ for all $\zeta\in V_{\zeta_0}$ and all $1\le j\le k$. Hence, the functions \eqref{13.1.17-} are constant in a neigborhood of $\zeta_0$.
\end{proof}

\begin{defn}\label{14.1.17'}For $\zeta\in \overline X$, let  $\lambda_1(\zeta),\ldots,\lambda_{m(\zeta)}(\zeta)$ be the eigenvalues of $A$, and let $\kappa_1(\zeta),\ldots,\kappa_{m(\zeta)}(\zeta)$ be the splitting amounts of $\lambda_1(\zeta),\ldots,\lambda_{m(\zeta)}(\zeta)$, respectively. Then, for $\zeta\in\overline X\setminus\mathrm{split\,}A$, we define
\[
\Theta_A(\zeta)=\begin{cases}
\big(\lambda_1(\zeta)-A(\zeta)\big)\cdot\ldots\cdot
\big(\lambda_{m(\zeta)}(\zeta)-A(\zeta)\big)\quad\qquad\qquad\text{if}\quad\zeta\in \overline X\setminus\mathrm{split\,}A,\\
\big(\lambda_1(\zeta)-A(\zeta)\big)^{\kappa_1(\zeta)}\cdot\ldots\cdot
\big(\lambda_{m(\zeta)}(\zeta)-A(\zeta)\big)^{\kappa_{m(\zeta)}}\quad\text{if}\quad \zeta\in\mathrm{split\,}A.
\end{cases}\]

\end{defn}
\begin{lem}\label{13.1.17---}
The function
\begin{equation}\label{13.1.17+}
\overline X\ni\zeta\longmapsto \Theta_A(\zeta)
\end{equation}is continuous on $\overline X$, and holomorphic in $X$.
\end{lem}
\begin{proof}
By Lemma \ref{15.1.17'}, for each $\xi\in \overline X\setminus \mathrm{split\,}A$, we have a neighborhood $U$ in $\overline X\setminus \mathrm{split\,}A$ of $\xi$, and functions $\lambda_1,\ldots,\lambda_m:U\to \C$ (uniquely determined up to the order), which are continuous on $U$ and holomorphic in $U\cap X$, such that, for each $\zeta\in U$, $\lambda_1(\zeta),\ldots,\lambda_m(\zeta)$ are the eigenvalues of $A(\zeta)$ and, therefore, $m=m(\zeta)$ and
\begin{equation*}
\Theta_{A(\zeta)}=\big(\lambda_1(\zeta)-A(\zeta)\big)\cdot\ldots\cdot\big(\lambda_m(\zeta)-A(\zeta)\big)\quad\text{for}\quad\zeta\in U.
\end{equation*}
This shows that \eqref{13.1.17+} is continuous on $\overline X\setminus \mathrm{split\,}A$, and holomorphic in $X\setminus \mathrm{split\,}A$.

It remains to prove that \eqref{13.1.17+} is continuous at each point of $\mathrm{split\,}A$. The holomorphy on $X$ then follows from the Riemann extension theorem, as   $X\cap\mathrm{split\,}A$ is a nowhere dense closed analytic subset of $X$ (by Theorem \ref{25.11.16}).

Since $X\cap\mathrm{split\,}A$ is nowhere dense in $X$ and $\partial X$ is $\Cal C^0$ smooth, it follows that $\mathrm{split\,}A$ is nowhere dense in $\overline X$. Therefore it is sufficient to prove that, for each $\xi\in\mathrm{split\,}A$ and each $\varepsilon>0$, there exists a neighborhood $U_\xi$ in $\overline X$ of $\xi$ such that
\begin{equation}\label{14.1.17'neu}
\big\Vert\Theta_A(\zeta)-\Theta_A(\xi)\big\Vert<\varepsilon\quad\text{if}\quad \zeta\in U_\xi\setminus\mathrm{split\,}A.
\end{equation}

Let $\xi\in \mathrm{split\,}A$ and $\varepsilon>0$ be given.

Let $\mu_1,\ldots,\mu_k$ be the eigenvalues of $A(\xi)$, and let $\kappa_j$ be the splitting amount of $\mu_j$. Since $A$ is continuous, then we can  find a neighborhood $U_\xi$ in $\overline X$ of $\xi$, and  $\delta>0$ so small that the disks $\D(\mu_j,\delta):=\{\lambda\in \C\,\vert\, \vert \lambda-\mu_j\vert<\delta\}$, $1\le j\le k$, are pairwise disjoint, and, for all $\zeta\in U_\xi$ and all $z^{(j)}_i\in\D(\mu_j,\delta)$, $1\le j\le k$, $1\le i\le \kappa_j$,
\begin{equation*}
\bigg\Vert\prod_{j=1}^k\prod_{i=1}^{\kappa_j}\Big(z^{(j)}_i-A(\zeta)\Big)-\prod_{j=1}^k\Big(\mu_j-A(\xi)\Big)^{\kappa_j}
\bigg\Vert<\varepsilon,
\end{equation*}and, hence,
\begin{equation}\label{14.1.17+}
\bigg\Vert\prod_{j=1}^k\prod_{i=1}^{\kappa_j}\Big(z^{(j)}_i-A(\zeta)\Big)-\Theta_A(\xi)
\bigg\Vert<\varepsilon.
\end{equation}
By Lemma \ref{12.1.17+}, shrinking $U_\xi$, we can moreover achieve that
\begin{itemize}
\item[-] for each $\zeta\in U_\xi$,  all eigenvalues of $A(\zeta)$ lie
in  $D(\mu_1,\varepsilon)\cup\ldots\cup D(\mu_k,\delta)$, and
\item[-]for each $\zeta\in U_\xi\setminus \mathrm{split\,}A$ and each $1\le j\le k$, exactly $\kappa_j$ of the eigenvalues of $A(\zeta)$, which we denote by $\lambda_1^{(j)}(\zeta), \ldots, \lambda_{\kappa_j}^{(j)}(\zeta)$, lie in $\D(\mu_j,\delta)$.
\end{itemize}
Then, by \eqref{14.1.17+}, for all $\zeta\in U_\xi\setminus\mathrm{split\,}A$,
\begin{equation*}
\bigg\Vert\prod_{j=1}^k\prod_{i=1}^{\kappa_j}\Big(\lambda^{(j)}_i(\zeta)-A(\zeta)\Big)-\Theta(\xi)
\bigg\Vert<\varepsilon.
\end{equation*}
where, by definition of $\Theta(\zeta)$,
\[
\prod_{j=1}^k\prod_{i=1}^{\kappa_j}\Big(\lambda^{(j)}_i(\zeta)-A(\zeta)\Big)=\Theta(\zeta).
\]which proves \eqref{14.1.17'neu}.
\end{proof}
By this lemma, slightly  modifying the proof of Theorem \ref{1.12.16'}, one obtains
\begin{thm}\label{14.1.17++} $X\setminus \mathrm{Jst\,}A$ is a nowhere dense analytic analytic subset of $X$,
$\overline X\setminus \mathrm{Jst\,}A$ is nowhere dense in $\overline X$,   and, if $\mathrm{Jst\,}A\not=\overline X$, then
there exist finitely many  functions $h_1,\ldots,h_\ell:\overline X\to \C$, which are continuous on $\overline X$, and holomorphic in $X$, such that \begin{equation}\label{14.1.17+++}\overline X\setminus\mathrm{Jst\,}A=\big\{h_1=\ldots= h_\ell=0\big\},\end{equation}
and
\begin{equation}\label{14.1.17+++neu}
\vert h_j(\zeta)\vert\le (2n)^{2n^4}\Vert A(\zeta)\Vert^{2n^4}\quad\text{for all}\quad \zeta\in \overline X\text{ and }1\le j\le \ell.
\end{equation}
\end{thm}


\begin{thebibliography}{10}


\bibitem [B1]{B1} H. Baumg\"artel, {\em Jordansche Normalform holomorpher Matizen}, Mber. Dt. Akad. Wiss, Berlin {\bf 11}, 23-24 (1969).

\bibitem[B2]{B2} H. Baumg\"artel, {\em Endlichdimensionale analytische St\"orungstheorie}, Akademie-Verlag Berlin, 1972.

\bibitem [B3]{B3} H. Baumg\"artel, {\em Analytic perturbation theory for linear operators depending on several complex variables}, Mat. Issled. {\bf 9}, 1, 17-39 (1974) (Russian).

\bibitem [B4]{B4} H. Baumg\"artel, {\em Analytic perturbation theory for matrices and operators},
 Birkhäuser, 1985.

\bibitem [Ga]{Ga} F. R. Gantmacher, {\em The theory of matrices. Volume one}, Chelesa, 1959.


\bibitem [FG]{FG} K. Fritzsche and  H. Grauert, {\em From holomorphic functions to complex manifolds}, Springer, 2002.



\bibitem [GF]{GF} H. Grauert and  K. Fritzsche, {\em Einf\"uhrung in die Funktionentheorie mehrerer Ver\"anderlicher}, Springer, 1974.

\bibitem [GH] {GH} I. Z. Gohberg and G. Heinig, {\em The resultant matrix and its generalizations. I: The resultant operator for matrix polynomials} (Russian), Acta Sci. Math. {\bf 37}, 41-61 (1975).


 Birkh\"auser, 2009.

\bibitem[GR]{GR} H. Grauert and R. Remmert, {\em Coherent analytic sheaves}, Springer, 1984.





\bibitem [KN] {KN} M. G. Krein and A. Naimark, {\em The method of symmetric and Hermitian forms in the theory of the
separation of the roots of algebraic equations}, Linear and Multilinear Algebra {\bf 10}, 265-308 (1981).

\bibitem[L]{L} S. \L ojasiewicz, {\em Introduction to complex analytic geometry}, Birkh\"auser, 1991.




\bibitem [Sh]{Sh} M. A.  Shubin, {\em Holomorphic familes of subspaces of a
Banach space}, Integral Equations and Operator Theory, {\bf 2}, no. 3, 407-420 (1979). Translation from {\em Mat. Issled.}
{\bf 5}, no. 4, 153-165 (1970) (Russian).


\bibitem [T]{T} G. P. A. Thijsse {\em Global holomorphic similarity to a Jordan form}, Results in Mathematics {\bf 8}, 78-87 (1985).

\bibitem[vdW]{vdW} B. L. van der Waerden, {\em Algebra I}, Springer 1971.

\bibitem [W]{W} W. Wasow, {\em On holomorphically similar matrices}, J. Math. Anal. Appl. {\bf 4}, 202-206 (1962).
\end{thebibliography}
\end{document}